\documentclass[14pt]{amsart}

\usepackage{amssymb}
\usepackage{amstext}
\usepackage{amsmath}
\usepackage{amscd}
\usepackage{latexsym}
\usepackage{amsfonts}
\usepackage{color}
\usepackage{enumerate}
\usepackage[all]{xy}
\usepackage{graphicx}
\usepackage[alphabetic,initials]{amsrefs}

\usepackage{hyperref}
\hypersetup{
   colorlinks=true, 
    linkcolor=blue, 
    urlcolor=green, 
}

\theoremstyle{plain}
\newtheorem{thm}{Theorem}[section]
\newtheorem*{thm*}{Theorem}
\newtheorem*{cor*}{Corollary}
\newtheorem*{defn*}{Definition}

\newtheorem{prop}[thm]{Proposition}
\newtheorem{lem}[thm]{Lemma}
\newtheorem{cor}[thm]{Corollary}

\newtheorem*{claim*}{Claim}

\theoremstyle{definition}
\newtheorem{defn}[thm]{Definition}

\newtheorem{rem}[thm]{Remark}

\theoremstyle{remark}

\begin{document}

\title[]{Stable Ulrich bundles on  cubic fourfolds }

\author[H.L. Truong]{Hoang Le Truong}
\address{Institute of Mathematics,  VAST,  18 Hoang Quoc Viet Road,  10307
Hanoi,  Viet Nam}
\address{Thang Long Institute of Mathematics and Applied Sciences,  Hanoi,  Vietnam}
\email{hltruong@math.ac.vn\\
	truonghoangle@gmail.com}

\author[H.N. Yen]{Hoang Ngoc Yen}
\address{The Department of Mathematics,  Thai Nguyen University of education.
20 Luong Ngoc Quyen Street,  Thai Nguyen City,  Thai Nguyen Province,  Viet Nam.}
\email{hnyen91@gmail.com}

\thanks{
The first author was partially supported by the Alexander von Humboldt Foundation and the Vietnam National Foundation for Science and Technology Development (NAFOSTED) under grant number 101.04-2019.309. The second author was partially supported by  Grant number  ICRTM02-2020.05,
awarded in the internal grant competition of International Center for Research and Postgraduate Training in Mathematics, Hanoi. }

\begin{abstract} 
		In this paper, we give necessary and sufficient conditions for the existence of Ulrich bundles on cubic fourfold $X$ of given rank $r$. As consequences,  we  show that  for every integer $r\ge 2$ there exists a family of indecomposable rank $r$ Ulrich bundles on the certain  cubic fourfolds, depending roughly on $r$ parameters, and in particular they  are of wild representation type; special surfaces on the cubic fourfolds  are explicitly constructed by Macaulay2; a new $19$-dimensional family of projective ten-dimensional irreducible holomorphic symplectic manifolds associated to a certain cubic fourfold is constructed; and for certain cubic fourfold $X$,  there exist arithmetically Cohen-Macaulay  smooth surface $Y \subset X$ which are not an intersection $X \cap T$ for a codimension two subvariety $T \subset \Bbb P^5$. 
		
\end{abstract}

\maketitle

\section{Introduction }

The study of moduli spaces of stable vector bundles of given rank and Chern classes on algebraic varieties is a very active topic in algebraic geometry.  In recent years attention has focused on ACM bundles,  that is vector bundles without intermediate cohomology.    Some of the reasons, why the study of ACM bundles  is important  are:
\begin{enumerate}[$\bullet$]
\item ACM bundles on the projective $n$-dimensional space  are precisely the bundles which are direct sum of line bundles (\cite{Hor64}).
\item The $i$-th syzygy of a resolution of any vector bundle on  hypersurfaces by split bundles,  is an arithmetically Cohen--Macaulay bundle (\cite{Eis80}).
\item These sheaves correspond to maximal Cohen--Macaulay modules over the associated coordinate ring of hypersurfaces (\cite{Bea00}). 
\item In \cite{KRR09},  ACM bundles on hypersurfaces have been used to provide counterexamples to a conjecture of Griffiths and Harris about whether subvarieties of codimension two of a hypersurface can be obtained by intersecting with a subvariety of codimension two of the ambient space.
\end{enumerate}

Note that there have been numerous studies of rank $2$ ACM bundles on surfaces and threefolds  (see  ( \cite{ArM09,  BrF09,  ChF09,  ChM05,  Mad00}) and the references in those papers),  
and a few studies of ACM bundles of higher rank \cite{ArG99, ArM09, Mad05,Tru19a,TrY20a}. There have also been a few examples of indecomposable ACM bundles of arbitrarily high rank \cite{PoT09}. But as far as we can tell,  examples of stable ACM bundles of higher ranks are essentially unknown. 

Very often the ACM bundles that we will construct and study will share another stronger property,  namely they have the maximal possible number of global sections; they will be the so-called Ulrich bundles. 
Notice that Ulrich bundles exist on curves,  linear determinantal varieties,  hypersurfaces,  and complete intersections. Moreover,  Ulrich bundles are semistable in the sense of Gieseker,  so once  fix rank and Chern class,  they may be parametrized 
by a quasi-projective scheme. In the case that $X$ is a hypersurface in $\Bbb P^n$ defined by a homogeneous form $f$,  Ulrich bundles on $X$ correspond to linear determinantal descriptions of powers of $f$ (\cite{Bea00}). This has been generalized to the case of arbitrary codimension in \cite{ESW03},  where Ulrich bundles are shown to correspond to linear determinantal descriptions of powers of Chow forms. It is clear by now that the concept of an Ulrich bundle encodes substantial algebraic information. We show in this article that in the case of a  cubic fourfold,  it also encodes substantial geometry. In particular,  we prove our first main result:

	\begin{thm}\label{main10}[cf. Theorem \ref{main1}]
Let $\mathcal F$ be  a vector bundle of rank $r\ge 3$ on a cubic fourfold $X$ in $\Bbb P^5$. Then the following statements are equivalent.
\begin{enumerate}[$1)$]
\item $\mathcal F$ is an Ulrich bundle of rank $r$. 
\item $\mathcal F$ is isomorphic to a vector bundle obtained from a certain surface $Y \subset X$ of degree $d = \frac{1}{2} (3r^2-  r)$ and sectional genus $r^3-2r^2+1$. 
\item There exists a  $3r\times 3r$ matrix $M$ of linear forms on $\Bbb P^5$ such that the  sequence
$$\xymatrix{0\ar[r]&\mathcal O_{\Bbb P^5}^{3r}(-1) \ar[r]^{M}&\mathcal O_{\Bbb P^5}^{3r} \ar[r]&\mathcal F\ar[r]&0.}$$
is exact. 
\end{enumerate}
\end{thm}

Recall that varieties that admit only a finite number of indecomposable ACM bundles (up to twist and isomorphism) are called {\it of finite representation type} (see \cite{DrG01} and references herein). Varieties of finite representation type have been completely classified into a short list in \cite[Theorem C]{BGS87} and \cite[p. 348]{EiH88}. They are three or less reduced points on $\Bbb P^2$,  a projective space,  a smooth quadric hypersurface $X\subseteq \Bbb P^n$,  a cubic scroll in $\Bbb P^4$,  the Veronese surface in $\Bbb P^5$ or a rational normal curve.

On the other extreme,  we would find the varieties {\it of  wild representation type} namely,  varieties for which there exist $t$-dimensional families of non-isomorphic indecomposable ACM sheaves for arbitrary large $t$.  In the case of dimension one,  it is known that curves of wild representation type are exactly those of genus larger than or equal to two. In dimension two and three,   Casanellas  and  Hartshorne  showed  in \cite{CaH11} that cubic surfaces and threefolds are of wild representation type. On the other hand,  the first result for varieties of arbitrary dimension were obtained in \cite{MiP14},  where the authors showed that Fano blow-ups of points in $\Bbb P^n$ are of wild representation type. In \cite{MiP13} it had already been proven that Bordiga or Castenuovo surfaces were of wild representation type.
 However, a broader problem has been much less studied: which are the possible dimensions of families of ACM irreducible sheaves on ACM projective schemes $X \subset \Bbb P^n$.  In \cite{BHP18}, authors showed that for every $r\ge0$, there exists a family of indecomposable ACM vector bundles of rank $r$, depending roughly on $r$ parameters, on certain projective surfaces. The main goal of this paper is to provide  families of ACM vector bundles on a large range of  a certain cubic fourfold, and in particular they are of wild representation type. 
%
Our source of examples will be the special cubic fourfold.
Recall that a smooth cubic fourfold  $X$ in $\Bbb P^5$ 
is  {\it special},  if there is  an embedding of a saturated rank-$2$ lattice 
$$L_\delta:=\langle h^2, Y\rangle\hookrightarrow A(X), $$
where $A(X)$ is the lattice of middle Hodge classes,  $h\in {\rm Pic}(X)$  is the hyperplane class,   $Y$ is an algebraic surface not homologous to a complete intersection,  and $\delta$ is the determinant of the intersection matrix of $\langle h^2, Y\rangle$. Note that  the locus  $\mathcal C_\delta$ of special cubic fourfolds  $X$ of a discriminant $\delta$ is an irreducible divisor which is nonempty if and only if $\delta>6$ and $\delta\equiv 0, 2 \pmod 6 $(\cite{Has00}).  
With the above notation, we have the following theorem. 
\begin{thm} \label{main3} There exists a cubic fourfold $X$ such that
\begin{enumerate}[$1)$]
\item $[X]\in\mathcal C_{14}\cap \mathcal C_{18}$. 
\item For any $r \ge 2$,  the moduli space of stable rank $r$ Ulrich bundles on a cubic fourfold $X$ is nonempty and smooth of dimension $r^2 + 1$. 
\item $X$ is of wild representation type.
\end{enumerate}
\end{thm}

Recently,   the study of some geometric properties of special ACM bundles on smooth cubic fourfolds   has received considerable attention.  In \cite{LLMS18},  the authors constructed an irreducible holomorphic symplectic eightfold $Z$ from the irreducible component of the Hilbert scheme of twisted cubic curves on 
 each cubic fourfold $X$ not containing a plane. Moreover,  $Z$ is isomorphic to an irreducible component of a moduli space of Gieseker stable torsion sheaves or rank three torsion free sheaves. In \cite{AdL17} it was shown that $Z$ is deformation equivalent to a Hilbert scheme of four points on a K3 surface. After that,   Lahoz et al. (\cite{LLMS18})  was shown that  $Z$ is birational to an irreducible component of a moduli space of Gieseker stable ACM bundles of rank six. In \cite{LMS17},  Lahoz,   Macr\`{i} and Stellari studied ACM bundles on cubic fourfolds containing a plane exploiting the geometry of the associated quadric fibration and Kuznetsov’s treatment of their bounded derived categories of coherent sheaves. More precisely,  they constructed the K3 surface naturally associated to the fourfold as a moduli space of Gieseker stable ACM bundles of rank four.  
 In view of above results, we search for  a moduli space of stable Ulrich bundles of rank $3$, which is a ten-dimensional irreducible holomorphic symplectic manifold. More precisely,  we are able to prove the next result.

\begin{thm} \label{main3}
Let  $X$ be a general special cubic fourfold in $\mathcal C_{18}$. Then
 the moduli space of stable rank $3$ Ulrich bundles on  $X$ with the Chern classes $c_1 = 3$ and $c_2 = 12$ is a smooth ten-dimensional holomorphically symplectic manifold.
\end{thm}

Let us explain briefly how this paper is divided. This paper is divided into 4 sections. In Section 2 let us briefly note  properties of ACM sheaves including the theory of matrix factorizations due to Eisenbud.
We shall prove Theorem \ref{main10} and its corollaries in Section 3. In Section 4,   as corollaries to the main theorem,  we give to explicit construction of  families of simple Ulrich vector bundles on a special cubic fourfold and  remarks on  a conjecture of Griffiths and Harris about whether subvarieties of codimension two of a hypersurface can be obtained by intersecting with a subvariety of codimension two of the ambient space. We construct a new $19$-dimensional family of projective ten-dimensional irreducible holomorphic symplectic manifolds.

\section{Preliminaries}

Throughout this section,   assume that $X$ is a closed  subscheme in $\Bbb P^n$,  
$R$ is the polynomial ring $k[x_0, \ldots,  x_n]$ over an algebraically closed field $k$ and $R_X$ is the coordinate ring of $X$. For any coherent sheaf $\mathcal F$,  we denote by $H_\bullet^i(X,\mathcal F)$ the sum $\bigoplus_{l\in\Bbb Z}H^i(X,\mathcal F(l))$.
Let us recall the definition of maximal Cohen-Macaulay $R_X$-modules and arithmetically Cohen-Macaulay coherent sheaves over $X$.
\begin{defn} A graded $R_X$-module $E$ is a maximal Cohen-Macaulay module (MCM),  if ${\rm depth} E = \dim E = \dim_RX$. A closed subscheme $X \subset \Bbb P^n_k$ is arithmetically Cohen-Macaulay (ACM),  if its homogeneous coordinate ring $R_X= R/I_X$ (where $I_X$ is the saturated ideal of $X$) is a Cohen-Macaulay ring.  This is equivalent to saying $H_\bullet^1(\mathcal{I}_{X, \Bbb P^n}) = 0$ and $H_\bullet^i(\mathcal{O}_X) = 0$,  for any $0 < i < \dim X$. A coherent sheaf $\mathcal E$ on  an ACM $X$ is an ACM sheaf,  if it is locally Cohen-Macaulay on $X$ and $H_\bullet^i(\mathcal E)=0$,  for any $0 < i < \dim X$.  

\end{defn}

Thanks to the graded version of the Auslander-Buchsbaum formula (for any finitely generated $R$-module $M$)
$${\rm pd}_R M=n+1-{\rm depth} M, $$
where ${\rm pd}_R M$ denotes the projective dimension of the $R$-module $M$. Moreover,
we deduce that a subscheme $X \subseteq \Bbb P^n$ is ACM if and only if ${\rm pd}(R_X) = {\rm codim} X$. Hence,  if $X \subseteq \Bbb P^n$ is a codimension $c$ ACM subscheme,  a graded minimal free $R$-resolution of $I_X$ is of the form:
$$\xymatrix{0\ar[r]&F_c\ar[r]^{\varphi_c}&F_{c-1}\ar[r]^{\varphi_{c-1}}&\ldots\ar[r]&F_1\ar[r]^{\varphi_1}&F_0\ar[r]&R_X\to0, }$$
where $F_0=R$ and $F_i=\bigoplus_{j}R(-i-j)^{\beta_{ij}(X)}$,  $1\le i\le c$. The integers $\beta_{ij}(X)$ are called the graded Betti numbers of $X$ and they are defined by
$$\beta_{ij}(X)=\dim_k{\rm Tor}^{i}(R_X, k)_{i+j}.$$
We construct the Betti diagram of $X$ writing in the $(i, j)$-th position the Betti number $\beta_{ij}(X)$. In this setting,  minimal means that ${\rm Im} \varphi_i \subset {\mathfrak{m}} F_{i-1}$. Therefore,  the free resolution of $X$ is minimal if,  after choosing a basis of $F_i$,  the matrices representing $\varphi_i$ do not have any non-zero scalar.

Notice that there is a one-to-one correspondence between ACM sheaves on  an ACM scheme $X$ and graded MCM $R_X$-modules that send $\mathcal E$ to $H^0_\bullet(X, \mathcal E)$(see \cite[2.1]{CaH11}). In the algebraic context,  MCM modules have been extensively studied (see for example the book of Yoshino \cite{Yos90}),  as they reflect relevant properties of the corresponding ring. There has also been recent work on ACM bundles of small rank on particular varieties such as Fano threefolds,  quartic threefolds and Grassmann varieties (see \cite{ArM09} and the references therein). When $X$ is a non-singular variety,  which is going to be mainly our case,  any coherent ACM sheaf on $X$ is locally free. For this reason,  we are going to speak often of ACM bundles (since we identify locally free sheaves with their associated vector bundle).

 It is well known that ACM sheaves provide a criterium to determine the complexity of the underlying variety. Indeed,  this complexity can be studied in terms of the dimension and number of families of indecomposable ACM sheaves that it supports. Recently,  inspired by an analogous classification for quivers and fork-algebras of finite type,  the classification of any ACM variety as being of finite,  tame or wild representation type(cf. \cite{DrG01} for the case of curves and \cite{CaH11} for the higher dimensional case) has been proposed. Let us recall the definitions:

 \begin{defn} An ACM scheme $X\subset \Bbb P^n$ is of finite representation type if it has,  up to twist and isomorphism,  only a finite number of indecomposable ACM sheaves. An ACM scheme $X\subseteq \Bbb P^n$ is of tame representation type if for each rank $r$,  the indecomposable ACM sheaves of rank $r$ form a finite number of families of dimension at most one. On the other hand,  $X$ will be of wild representation type if there exist $t$-dimensional families of non-isomorphic indecomposable ACM sheaves for arbitrary large $t$.
 \end{defn}

 Varieties of finite representation type have been completely classified into a short list in \cite[Theorem C]{BGS87} and \cite[p. 348]{EiH88}. They are three or less reduced points on $\Bbb P^2$,  a projective space,  a non-singular quadric hypersurface $X\subseteq \Bbb P^n$,  a cubic scroll in $\Bbb P^4$,  the Veronese surface in $\Bbb P^5$ or a rational normal curve. The only known example of a variety of tame representation type is the elliptic curve. On the other hand,  so far only a few examples of varieties of wild representation type are known: curves of genus $g\ge2$ (cf. \cite{CaH11}),  del Pezzo surfaces and Fano blow-ups of points in $\Bbb P^n$(the cases of the cubic surface and the cubic threefold have also been handled in \cite{CaH11}) and ACM rational surfaces on $\Bbb P^4$(\cite{MiP13}).
 
 Very often the ACM bundles that we will construct will share another stronger property,  namely they have the maximal possible number of global sections; they will be the so-called Ulrich bundles. Let us end this section recalling the definition of Ulrich sheaves and summarizing the properties that they share and that will be needed in the sequel.
\begin{defn}
A coherent sheaf $\mathcal F$ on $X$ is said to be {\it initialized}, if
$$H^0(X, \mathcal F(-1))=0\text{ and } H^0(X, \mathcal F)\neq 0.$$
If $\mathcal F$ is ACM,  then there exists an integer $k$ such that  $\mathcal F_{init}:=\mathcal F(k)$  is initialized. 
\end{defn}
\begin{defn}
Given a projective scheme $X \subseteq \Bbb P^n$ and a coherent sheaf $\mathcal F$ on $X$.
A vector bundle $\mathcal F$ is an Ulrich bundle, if $\mathcal F$ is ACM on $X$ and the initialized twist $\mathcal F_{init}$ of $\mathcal F$ satisfies $h^0(X, \mathcal F_{init})={\rm rank}(F)\deg(X)$.
\end{defn}

   When $X \subseteq \Bbb P^n$ is a hypersurface of degree $s$,  the existence of Ulrich bundles is related to classical problems in algebraic geometry (\cite{Bea00}). If ${\rm rank}(\mathcal F) = 1$,  then one has a determinantal presentation of $X:= \{\det(M) = 0\}$,  where $M = (\ell_{ij})$,  $1\le i, j\le s$ is a matrix of linear forms; a bundle $\mathcal F$ with ${\rm rank}(\mathcal F) = 2$ corresponds to a Pfaffian equation of $X : \{{\rm pf}(M) = 0\}$,  where $M$ is a $(2s) \times (2s)$ skew-symmetric linear matrix(see \cite{Bea00}). In this paper,  we describe some criterions for determining when a cubic fourfold $X$ has an Ulrich bundle of given rank $r$. It is clear that each direct summand of an Ulrich bundle is also Ulrich. Thus,  one can restrict the attention to indecomposable bundles,  i.e. bundles that do not split as direct sum of bundles of smaller ranks. The study of indecomposable Ulrich bundles is a particularly in problem that could give some suggestions on the complexity of the embedding $X\subset \Bbb P^N$.  
  
%


 
Now let $X = V(f) \subseteq \Bbb P^n$ be a hypersurface cut out by a homogeneous form $f$ of degree $s$. It is well-known that a matrix factorization $(A,  B)$ of $f$ induces a maximal Cohen-Macaulay module supported on $X$ by ${\rm coker} A$. Conversely,  if we have a maximal Cohen-Macaulay $R_X = R/(f)$-module,  one has a matrix $A$ by reading off its length $1$ resolution. Indeed,  it forms a part of a matrix factorization of $f$. Therefore,  there is a unique matrix $B$ such that $AB = BA = f\cdot \mathrm{Id}$. As a conclusion,  there is a bijection between the maximal Cohen-Macaulay modules and the equivalence classes of matrix factorizations of $f$.

 Now,  let us briefly recall Shamash’s construction (see \cite{Sha69}). Suppose that  $Y$ is a subscheme contained in hypersurface $X=V(f)\subseteq \Bbb P^n$. Let  $R_Y$ be the coordinate ring of $Y$,  and $F^*$  the minimal free $R$-resolution of $R_Y$. Since $Y \subseteq X$,  we have a right exact sequence
$$\xymatrix{\ldots \ar[r]&F_1 \otimes_R R_X\ar[r] &F_0 \otimes_R  R_X\cong R_X\ar[r]&R_Y \ar[r]&0.}$$
Hence,  there is an $R$-free resolution of $R_Y$ (possibly non-minimal)
$$\xymatrix{\ldots \to  G_4\oplus G_2(-s)\oplus G_0(-2s) \to G_3\oplus G_1(-s)\to  G_2\oplus G_0(-s) \to G_1 \to G_0 \to R_Y \to 0}, $$
where $G_i = F_i \otimes_R R_X$. It becomes eventually $2$-periodic after the $(c -1)$-th step,  where $c={\rm codim} R_Y$. Hence,  it induces a matrix factorization of $f$. Such a matrix factorization provides a presentation of an ACM sheaf on $X$. 
Moreover,  this construction allows us to control  some extent the degrees of the entries of the corresponding minimal matrix factorization of $f$ induced by an $R_X$-module $R_Y$,  if we know the Betti numbers of $R_Y$ as an $R$-module. Thus,  the Betti numbers of the minimal periodic resolution is called  the {\it shape} of the matrix factorization.


Let $X$ be a  cubic fourfold containing a quintic Del Pezzo   surface $T$. Then,  $T$ have the minimal free resolution
\begin{equation}\label{eqn}
\xymatrix{0\ar[r]& \mathcal{O}_{\Bbb P^5}(-5)\ar[r]& \mathcal{O}_{\Bbb P^5}^{5}(-3)\ar[r]&\ar[r]\mathcal{O}^{5}_{\Bbb P^5}(-2)\ar[r]&\mathcal{I}_{T}\ar[r]&0.}
\end{equation}
%
An easy way to produce matrix factorizations on a hypersurface $X = V(f)$ in $\Bbb P^5$ is to consider  $\Gamma_\bullet(\mathcal{O}_T)$ as a module over $R = K[x_0, \ldots, x_5]$ annihilated by $f$. A matrix factorization of $f$ is given by the periodic part of a minimal free resolution of $\Gamma_\bullet(\mathcal{O}_T)$ as a module over $R_X = R/(f)$. In particular,   the minimal resolution of $\Gamma_\bullet(\mathcal{O}_T)$ as a module over the homogeneous coordinate ring of a cubic fourfold $X \supset T$ is eventually $2$-periodic with Betti numbers
		\begin{center}
		\begin{tabular}{ c | c c c c c c c c c}
			&    $0$ & $1 $ &$2 $  &$3 $& $4 $ &$5 $  &$6 $\\ 
			\hline
			0&  $1    $ & $\cdot$ & $\cdot$ & $\cdot$ & $\cdot$ & $\cdot$ & $\cdot$\\ 
			1&   $\cdot$ & $5    $ &$6$&$\cdot$ & $\cdot$ & $\cdot$ & $\cdot$ \\ 
			2&   $\cdot   $ & $\cdot    $ &$  \cdot $  & $  6  $ & $6$ & $\cdot$ & $\cdot$ \\ 
			3&   $\cdot$ & $\cdot$ &$  \cdot  $  & $  \cdot  $ & $  \cdot $ & $  6  $ & $6$ 

		\end{tabular}
		
	\end{center}
Hence,   $X$ has a matrix factorization of type $(\psi : \mathcal{O}^{6}_X(-3)  \to 	\mathcal{O}^6_X(-1); \varphi: \mathcal{O}^6_X(-1) \to \mathcal{O}_X^{6})$.  Let $\mathcal F={\rm coker} \varphi$. Then $\mathcal F$ is an Ulrich bundle of rank $2$.

To each $Y$ as above,  we associate the vector bundle $\mathcal F$ by Serre’s construction:
$$\xymatrix{0\ar[r]& \mathcal O_X(-2)\ar[r]^s&\mathcal F \ar[r]&\mathcal I_{Y/X}\ar[r]&0, }$$
where $\mathcal I_{Y/X}$ is the ideal sheaf of $Y$ in $X$,  and $s$ is a section of $\mathcal F$ such that $Y$ is its zero locus. Let $H$ be the class of a hyperplane section of $X$. Then,  we have $c_1(\mathcal F)=2H$ and $c_2(\mathcal F)=5$. 
As follows from \cite{Bea00, Bea02},  the vector bundles $\mathcal F$ of this type have several other equivalent characterizations:

	\begin{thm}\label{class100}
Let $X$ be a  cubic fourfold in $\Bbb P^5$ and $\mathcal F$ a rank $2$ vector bundle on $X$. Then the following statements are equivalent.
\begin{itemize}
\item[$1)$] $\mathcal F$ is an Ulrich bundle of rank $2$. 
\item[$2)$] $\mathcal F$ is isomorphic to a vector bundle obtained by Serre’s construction from a quintic Del Pezzo   surface $T\subset X$.
\item[$3)$] There exists a skew-symmetric $6$ by $6$ matrix $M$ of linear forms on $\Bbb P^5$ such that $\mathcal F$ considered as a sheaf on $\Bbb P^5$ can be included in the following exact sequence:
$$\xymatrix{0\ar[r]&\mathcal O_{\Bbb P^5}^{6}(-1) \ar[r]^{\varphi}&\mathcal O_{\Bbb P^5}^{6} \ar[r]&\mathcal F\ar[r]&0.}$$
\end{itemize}

\end{thm}


\section{Generalities on Ulrich bundles }
The goal of this section is to extend Theorem \ref{class100} from \cite{Bea00, Bea02}  to Ulrich bundles of  rank $r\ge3$ on cubic fourfolds. First, we show how to associate an ACM surface to an Ulrich bundle on cubic fourfolds from $\mathcal N$-type resolutions (or so-called Bourbaki sequences,  c.f. \cite{BHU87}). We recall the definition here.
\begin{defn} Let $X \subset \Bbb P^n$ be an equidimensional scheme,  and let $Y \subset X$ be a codimension $2$ subscheme without embedded components. 
A coherent sheaf $\mathcal L$ on $X$ is {\it dissoci\'{e}} if it is isomorphic to a direct sum $\bigoplus\mathcal{O}_X(a_i)$  for various $a_i \in\Bbb Z$.
An $\mathcal N$-type resolution of $Y$ is an exact sequence
$$\xymatrix{0\ar[r]&\mathcal L\ar[r] &\mathcal N \ar[r] &\mathcal{J}_{Y, X} \ar[r] &0}, $$
with $\mathcal L$ is dissoci\'{e},  and $\mathcal{N}$ is a coherent sheaf satisfying $H_\bullet^1(\mathcal N^\vee) = 0$,  and ${\rm Ext}^1(\mathcal N, \mathcal{O}_X) = 0.$
\end{defn}
In the case that $X$ satisfies Serre’s condition $S_2$,  and $H_\bullet^1(\mathcal{O}_X) = 0$,  then an $\mathcal N$-type resolution of $Y$ exists (see \cite[2.12]{Har03}).  
An $\mathcal N$-type resolution is not unique but it is well known that any two $\mathcal N$-type resolutions of the same subscheme are stably equivalent (see \cite[1.10]{Har03}). In other words,  if $\mathcal N$,  and $\mathcal N^\prime$ are two sheaves appearing in the middle of an $\mathcal N$-type resolution of a subscheme $Y$,  then there exist dissoci\'{e} sheaves $\mathcal L_1$,  $\mathcal L_2$,  and an integer $a$ such that
$$\mathcal N \oplus \mathcal L_1 \cong \mathcal N^\prime(a)\oplus \mathcal L_2.$$
The following theorem and its converse involve the relationship between Ulrich bundles and ACM surfaces.
\begin{thm}\label{thm4.5}
Let   $\mathcal F$ be a  Ulrich bundle of rank $r \ge 2$  on a cubic fourfold $X$ and generated by its global sections. 
	Then,  for every $a\ge 1$,  there  is  a surface $Y \subset X$  such that  $Y$ has an $\mathcal N$-type resolution
$$\xymatrix{0\ar[r]& \mathcal{O}_X^{r-1}(-a)\ar[r]&\ar[r]\mathcal{F}\ar[r]&\mathcal{I}_{Y/X}(b)\ar[r]&0}$$
for some $b\in\Bbb Z$.	

\end{thm}
\begin{proof}
%
%

We denote by $M$ the module $H_\bullet^0(\mathcal F)$. Then,  we have a short exact sequence
$$\xymatrix{0\ar[r]&R^{3r}(-1)\ar[r]&  R^{3r}\ar[r]& M\ar[r]&0.}$$
Therefore,  the Castelnuovo-Mumford regularity of $M$ is $0$.
 Let $N$ be a graded submodule of $M$ generated by homogeneous elements of $M$ all of
the same degree $a-1$. It follows from the Castelnuovo-Mumford regularity of $M$ is $0$  and $a\ge 1$ that  $M/N$ is of finite length over $R$. We now observe that $N$ has a presentation 
$$\xymatrix{R^{c}(-a) \ar[r]^{\phi}& R^{b}(-a+1)\ar[r] &N\ar[r]&0, }$$ 
 in which the presenting matrix $\phi$ has homogeneous entries with all the entries having degree $1$.  We then adjoin indeterminates $Y_{ij}$ to $R$ for $1 \le i \le b$,  and $1 \le j \le 5$
to obtain a new ring $A = R[Y_{ij}]$. We put $M_1 = A \otimes_RM$,  and $N_1 = A \otimes_RN$. We may 
consider the matrix $\phi$ as an $A$-homomorphism $A^{c}(-a) \to A^{b}(-a+1)$ presenting $N_1$,  that is,  we are identifying $A \otimes \phi$ with $\phi$. Let $\Phi$ be the matrix with elements $Y_{ij}$,  we obtain an
$A$-homomorphism $A^{c}(-a) \oplus A^{3r}(-a) \to A^{b}(-a+1)$ via the matrix $\Psi= [\phi|\Phi]$,  where the vertical line represents a matrix partitioning corresponding to the two direct summands of
$A^{c}(-a) \oplus A^{3r}(-a)$. Let $E$ be the image of $A^{3r}(-a)$ in $N_1$ under the composition of the maps
$$\xymatrix{A^{c}(-a) \oplus A^{3r}(-a) \ar[r]^{\quad \quad\ \Psi}& A^{b}(-a+1) \ar[r]& N_1.}$$ Then,  $E$ is a free submodule of $N_1$ of rank ${r}$,  and also $E$ is a graded submodule of $N_1$. 
Since $M$ is Cohen-Macaulay as $R_X$-module,   $M_1$ is Cohen-Macaulay as $A$-module.
 It follows from $M_1/E$ being a graded $A$-module that its associated prime ideals are homogeneous. Since $\dim A\ge \dim R=6$,  and the $(x_0, \ldots, x_5)$-depth of $M$ over $R$ is equal to the $(x_0, \ldots, x_5)$-depth of $M_1$ over $A$,  we get that ${\rm depth}_{(x_0, \ldots, x_5)}(M_1/E)=4$. Hence,  because our field is infinite,  there must be a form $g$ of degree one in $x_0,  ..., x_5$ that is not a zero divisor on $M_1/F$. 

Next we show that $A_g \otimes_A(N_1/E)$ is isomorphic to a normal prime ideal $I_g$ of $A_g$. Indeed,   since  degree of $g$ is one in $x_0, \ldots,  x_5$,  we have that $M_g$ at $N_g$ is a free $A_g$-module. Hence,  $(N_1)_g = A_g \otimes_A N_g$ is also a free $A_g$-module. Thus,  the sequence
$$\xymatrix{A_g^{c}(-a) \ar[r]^{\Psi}& A_g^{b}(-a+1)\ar[r] &(N_1)_g\ar[r]&0}$$ 
splits. After performing invertible row and column operations on $\Psi$,  we obtain a $b\times c$ matrix of the form
\begin{center}\begin{tabular}{ |c | c| }
\hline
$\bf 0$ &$\bf{I}_t$\\
\hline
$\bf 0$&$\bf0$\\
\hline
\end{tabular}
\end{center}
where  $\bf I_t$ is the $t\times t$ identity matrix for some $t$. The effect of the row operations on the matrix 
$\Psi= [\phi|\Phi]$,  is to obtain a matrix $\Psi^\prime= [\phi^\prime|\Phi^\prime]$,  where $\Phi^\prime$ consists of entries $y^\prime_{ij}$ which are a new set of ${3r}\times b$ polynomial indeterminates over $R_g$. After applying further invertible column operations on $\Psi$,  we obtain  
\begin{center}$\Psi^\prime$=\begin{tabular}{ |c | c  | c|}
\hline
$\bf 0$ &$\bf{I}_t$&$\bf 0$\\
\hline
$\bf 0$&$\bf0$&$y^{\prime\prime}_{ij}$\\
\hline
\end{tabular}
\end{center}
where the $y^{\prime\prime}_{ij}$ in the lower right-hand corner are a subset of the $y^{\prime}_{ij}$. Let $I_g$ be the $A_g$-ideal  of projective dimension one, that is, the ideal generated by the maximal minors of the matrix $(y^{\prime\prime}_{ij})$ in the lower right-hand corner of $\Psi^\prime$. Then,  $A_g \otimes_A(N_1/E)\cong I_g$.   Since $g$ is regular on $M_1/E$,  we get that $M_1/E\cong I$. Since this matrix consists of indeterminates over $R_g$,  it follows that $I_g$ is a normal prime ideal of $A_g$,  and so $I$ is a normal prime ideal of $A$. Therefore,  we have an exact sequence
$$\xymatrix{0\ar[r] &A(-a)^{3r}\ar[r] & M_1\ar[r]&I\ar[r]&0, }$$
where $I$ is a homogeneous prime ideal of $A$ of height two,  and $A/I$ is Cohen-Macaulay. Then,  it follows from Flenner's version of Bertini's Theorem (see \cite[Section 5]{Fle77}) that there exists a  linear regular sequence $h_1,  \ldots, h_{{3r}b}$ on $A/I$. It follows that the sequence 
$$\xymatrix{0\ar[r] &(A/(h_1, \ldots, h_{{3r}b})A)(-a)^{3r}\ar[r] & M_1/(h_1, \ldots, h_{{3r}b})M_1\ar[r]&I/(h_1, \ldots, h_{{3r}b})I\ar[r]&0}$$
is exact. Since $M_1/(h_1, \ldots, h_{{3r}b})M_1\cong M$ and $A/(h_1, \ldots, h_{{3r}b})A\cong R$,   there  is  a codimension $2$ subscheme  $Y \subset X$ without embedded components such that  $Y$ has an $\mathcal N$-type resolution
$$\xymatrix{0\ar[r]& \mathcal{O}_X^{r-1}(-a)\ar[r]&\ar[r]\mathcal{F}\ar[r]&\mathcal{I}_{Y/X}(b)\ar[r]&0}$$
for some $b\in\Bbb Z$.	

\end{proof}	
%

\begin{prop}\label{exY}
Let   $\mathcal F$ be an Ulrich bundle of rank $r \ge 3$ on the cubic fourfold $X$ and generated by its global sections. 
	Then there  is  an ACM surface $Y \subset X$ of degree $d = \frac{1}{2} (3r^2-  r)$ and sectional genus $r^3-2r^2+1$ 
	such that  $Y$ has an $\mathcal N$-type resolution
\begin{equation}\label{eqn3}
\xymatrix{0\ar[r]& \mathcal{O}_X^{r-1}\ar[r]&\ar[r]\mathcal{F}\ar[r]&\mathcal{I}_{Y/X}(r)\ar[r]&0.}
\end{equation}
Moreover,   $Y$ has the minimal free resolution
\begin{equation}\label{eqn4}
\xymatrix{0\ar[r]& \mathcal{O}_{\Bbb P^5}^{r-1}(-r-3)\ar[r]& \mathcal{O}_{\Bbb P^5}^{3r}(-r-1)\ar[r]&\ar[r]\mathcal{O}^{2r+1}_{\Bbb P^5}(-r)\oplus \mathcal O_{\Bbb P^5}(-3)\ar[r]&\mathcal{I}_{Y}\ar[r]&0.}
\end{equation}
%
\end{prop}
\begin{proof}

Since $\mathcal F$ is an Ulrich bundle of rank $r$,  we can take $r - 1$ general sections of $\mathcal F$ so that the quotient of $\mathcal F$ by $\mathcal{O}_X^{r-1}$ is torsion free.
Therefore we have an exact sequence
$$\xymatrix{0\ar[r]& \mathcal{O}_X^{r-1}\ar[r]&\ar[r]\mathcal{F}\ar[r]&\mathcal{I}_{Y/X}(D)\ar[r]&0}
$$where $D = c_1(\mathcal F)=rH$ is a certain divisor on $X$ and $Y$ is a surface of degree
equal to $c_2(\mathcal F)$.
As $\mathcal F$ is generated by its global sections,  $\mathcal{I}_{Y/X}(r)$ is also generated by global sections and $h^0(\mathcal{I}_{Y/X}(r)) = 2r + 1$.
On the other hand,  $\mathcal F$ has the following minimal free resolution in $\Bbb P^5$
$$\xymatrix{0\ar[r]&\mathcal O_{\Bbb P^5}^{3r}(-1) \ar[r]^{\varphi}&\mathcal O_{\Bbb P^5}^{3r} \ar[r]&\mathcal F\ar[r]&0}$$
by \cite[Proposition 3.7]{CaH11}. Using the mapping cone procedure with the free resolutions of $\mathcal{O}_X$ and the exact sequence \ref{eqn3}, $Y$ has the minimal free resolution as required. 
Therefore $Y$ is an ACM surface.
By \cite[Proposition 3.7]{CaH11},  we have $\deg Y = \frac{1}{2} (3r^2-  r)$ and by Riemann--Roch,  sectional genus of $Y$ is $r^3 - 2r^2 + 1$.

\end{proof}

\begin{rem}
There have been many studies about ACM bundles of rank two in dimensions two and three  in  \cite{ArM09,  BrF09,  ChF09,  ChM05,  Mad00}. The higher rank case has been investigated in \cite{ArG99, ArM09, Mad05,Tru19a,TrY20a}. The papers \cite{MiP14} and \cite{PoT09} give a few examples of indecomposable ACM bundles of arbitrarily high rank. The already mentioned papers \cite{CaH11,  CaH12} contain a systematic study of stable ACM bundles in higher rank on cubic surfaces and threefolds. A general existence result for Ulrich bundles on hypersurfaces is given in \cite{BHU91}. But in general,  the existence of Ulrich sheaves on a projective scheme $X$ is not known.  
\end{rem}
\begin{rem}
Liaison has become an established technique in algebraic geometry.  The greatest activity,  however,  has been in the last quarter-century  beginning with the work of Peskine and Szpir\'{o}  in 1974(see \cite{PeS74}). Liaison is a powerful tool for constructing examples.

\begin{defn} We say $Y_2$ is obtained by an elementary biliaison of height $m$ from $Y_1$,  if there exists an ACM scheme $Z \subset \Bbb P^n$,  of dimension $r + 1$ containing $Y_1$ and $Y_2$,  such that $Y_2 \thicksim Y_1 + mH$ on $Z$ . (here $\thicksim$ means linear equivalence of divisors on $Z$ in the sense of \cite{Har94}.) The equivalence relation generated by elementary biliaisons is called simply biliaison. If $X$ is a complete intersection scheme in $\Bbb P^n$,  we speak of CI-biliaison. If  $Y_1$,  $Y_2$,  $Z $ are contained in some projective scheme $X \subseteq\Bbb P^n$,  we speak of biliaison  (CI-biliaison) on $X$.
\end{defn}

Casanellas and Hartshorne gave a criterion for when two codimension $2$ subschemes $Y_1$,  $Y_2$ of a normal ACM projective scheme $X$ are in the same biliaison class (see \cite[Theorem 3.1]{CaH04}). As application,  we have the following corollary.
\end{rem}
\begin{cor}
Let $\mathcal F$ be an Ulrich bundle on a cubic foufold $X$. Then   surfaces  constructed from $\mathcal F$ as in Theorem \ref{thm4.5} 
are ACM and belong to the same $CI$-biliaison equivalence class on  $X$.
\end{cor}
\begin{proof}
It follows immediately from Theorem \ref{thm4.5}, Theorem 3.1 in \cite{CaH04}) and Proposition \ref{exY}
\end{proof}	

Now we will prove a converse to the Proposition \ref{exY}.
\begin{thm} \label{thm4.2}
  Assume that  $Y$ is  an ACM smooth surface with the minimal free resolution  \ref{eqn4} for $r\ge 3$. Let $X$  be a general cubic fourfold containing $Y$. Then $Y$ has an $\mathcal N$-type resolution with $\mathcal F$ being an Ulrich bundle of rank $r$
$$\xymatrix{0\ar[r]& \mathcal{O}_X^{r-1}\ar[r]&\ar[r]\mathcal{F}\ar[r]&\mathcal{I}_{Y/X}(r)\ar[r]&0.}$$

\end{thm}
\begin{proof} 
Consider $\Gamma_\bullet(\mathcal{O}_Y)$
	as a $R_X$-module,  the periodic part of its minimal free resolution
	yields,  up to twist,  matrix factorization of the form
	$$\xymatrix{ R^{3r}(-3)\ar[r]^{\psi}& R^{3r}(-1)\ar[r]^{\quad\varphi} &R^{3r}.}$$	
Let $F^\bullet$  and $\overline{G^\bullet}$ be minimal free resolutions of the section ring $\Gamma_\bullet(\mathcal{O}_Y)$ as an $R$-module,  and $R_X$-module, respectively. Let $\varphi$ be the  map $\overline{G_4} \to \overline{G_3}$,  and $\mathcal{F} = \widetilde{{\rm coker} \varphi}(-3)$.
Since $\deg \det \varphi=1\cdot (3r)=3\cdot r$, 
$\mathcal{F}$ is an ACM sheaf of rank $r$,  and so the Shamash resolution is
$$\to\mathcal{O}_{X}^{3r}(-r-4)\to \mathcal{O}_{X}^{3r}(-r-3)\to \mathcal{O}_{X}^{3r}(-r-1)\oplus \mathcal O_{X}(-3)\to\mathcal{O}^{2r+1}_{X}(-r)\oplus \mathcal O_{X}(-3)\to\mathcal{I}_{Y/X}\to0.$$
%
%
	Thus,  the Shamash resolution is always non-minimal,  and in a minimal resolution a cancellation occurs,  causing $\beta_{1, 3}=8$, if $r=3$,  and $\beta_{1, 3}=1$, otherwise to decrease by one. Such cancellation corresponds to the equation $f$ of $X$ in $R$ becoming superfluous in $R_X=R/(f)$. 
	Since $\mathcal{F} = \widetilde{{\rm coker} \varphi}(-3)$,   the sheaf $\mathcal F$ is the sheafification of a MCM module over $R_X$. Thus,  the sheaf $\mathcal F$ is an Ulrich  bundle of rank $r$. 
	By the definition of the Shamash resolution,  the map $\overline{G_3} \to \overline{G_2}$ factorizes through $\mathcal{F}$. Thus,  we obtain a map $\alpha:\mathcal{O}_X^{r-1}(-r)\to \mathcal{F}$.

Now, we apply the functor ${\rm Hom}(\bullet,  \omega_X )$ to $\alpha$,  and obtain
$ \alpha^*(-r+1):\mathcal{F}^*(-r+1)\to \mathcal{O}_X^{r-1}$. The cokernel of this map coincides by construction with the cokernel of the dual of the sheafification of the last map of $F^\bullet$
$$ \mathcal{O}_X^{3r}(-r+1)\to \mathcal{O}_X^{r-1}, $$
which is a presentation of $\omega_Y$ by duality on $\Bbb P^5$. Since ${\rm rank} \mathcal F= r=(r-1)+1$,  both $\alpha^*(-r+1)$,  and $\alpha$ drop rank in expected codimension $2$. Applying again ${\rm Hom}(\bullet,  \omega_X )$ to $\alpha^*(-r+1)$ we get that $\alpha$ is injective. It follows from the Hilbert-Burch Theorem that  we have an exact sequence
	$$\xymatrix{0\ar[r]&\mathcal{O}_X^{r-1}(-r)\ar[r]^{\alpha}&\ar[r]\mathcal{F}&\ar[r]\mathcal{O}_X(\ell)& \mathcal{O}_Y(\ell)\ar[r]&0}, $$
for some $\ell$. By applying again ${\rm Hom}(\bullet,  \omega_X )$ to this last exact sequence one gets that $\alpha^*(-r)$ is a presentation of $\omega_Y(-\ell)$. Hence,  we have $\ell = 0$,  as required.
\end{proof}

	\begin{thm}\label{main1}
Let $\mathcal F$ be  a vector bundle of rank $r\ge 3$ on a cubic fourfold $X$ in $\Bbb P^5$. Then the following statements are equivalent.
\begin{enumerate}[$1)$]
\item $\mathcal F$ is an Ulrich bundle of rank $r$. 
\item $\mathcal F$ is isomorphic to a vector bundle obtained from a surface $Y \subset X$ 
	 such that $\mathcal{I}_{Y}$ has the minimal free resolution \ref{eqn4}.

\item There exists a  $3r\times 3r$ matrix $M$ of linear forms on $\Bbb P^5$ such that the  sequence
$$\xymatrix{0\ar[r]&\mathcal O_{\Bbb P^5}^{3r}(-1) \ar[r]^{M}&\mathcal O_{\Bbb P^5}^{3r} \ar[r]&\mathcal F\ar[r]&0.}$$
is exact.
\end{enumerate}
\end{thm} 
\begin{proof}
$1)\Rightarrow 2)$: This  follows from Proposition \ref{exY}.\\
$2)\Rightarrow 3)$: This follows from Theorem \ref{thm4.2}.\\
$3)\Rightarrow 1)$: This is trivial.
\end{proof}


\section{ Ulrich bundles of small rank}

First, let us describe a little more details on the existence of Ulrich bundles on smooth cubic fourfolds. In \cite[Proposition 2.5]{KiS20}, Kim and Schreyer showed that if $\mathcal F$ is be an Ulrich bundle of rank $r$ on a {\it very general cubic fourfold} then $r$ is divisible by 3 and $r \ge 6$. 
Using invariant theory, Manivel \cite[ Propostion 2.2]{Man19}  found a family of {\it very general cubic fourfolds}  having an Ulrich sheaf of rank $9$. In \cite{Bea00}, Beauville showed that there exists a smooth cubic fourfold having an Ulrich bundle of rank $2$. The set of such cubic fourfolds contains  an open (but it is not open) subset of the divisor $\mathcal C_{14}$(\cite{BRS19}).  Hence, it is natural to ask the existence question for rank $3$ Ulrich bundles on smooth cubic fourfolds. The goal of this section is to prove Theorem \ref{main2}, which  answers this question.

%

Now, let $\mathcal H_{9, 4}$ be the component of the Hilbert scheme of curves of degree
$d = 9$ and genus $g = 4$ in $\Bbb P^4$ which dominate the moduli spaces $\mathcal M_4$ of curves of genus $4$.

\begin{prop}\label{C3.2}
	Let $C\in \mathcal H_{9, 4}$ be a general point.				 Then $C$ is of maximal rank and	
	homogeneous coordinate ring $R_C=R/I_C$ and the section ring $\Gamma_\bullet(\mathcal{O}_C)$ have minimal free resolutions with the following	Betti tables:
\begin{center}
		\begin{tabular}{ c | c c c c c c c}
			&   & $0$ & $1 $ &$2 $  &$3 $&$4$\\ 
			\hline
			0&  & $1    $ & $\cdot$ & $\cdot$ & $\cdot$&$\cdot$ \\ 
			1&  & $\cdot    $ & $\cdot$ & $\cdot$ & $\cdot$&$\cdot$\\ 
			2&  & $\cdot$ & $11$ &$18$ &$9$&$1$\\ 
		\end{tabular}	
		\ \ \
			\begin{tabular}{ c | c c c c c c}
			&    $0$ & $1 $ &$2 $  &$3 $\\ 
			\hline
			0&   $1    $ & $\cdot$ & $\cdot$ & $\cdot$ \\ 
			1&   $1$ & $5$ &$\cdot$ &$\cdot$ \\ 
			2&   $\cdot$ & $1   $ &$  8 $ &$  4  $ \\ 
		\end{tabular}

\end{center}
	\end{prop}
\begin{proof}
	Assuming that the restriction map $H^0(\mathcal{O}_{\Bbb P^4}(m))\to H^0(\mathcal{O}_{C}(m))$ has maximal rank. Then we have the following statements.
	\begin{itemize}

\item	The Hilbert series of the homogeneous coordinate ring of $C$ is
	$$H_C(t) = 1 + 5t + 15t^2 + 24 t^3  + 33 t^4  + 42 t^5  +11 t^6 +60 t^7 +\ldots.$$ 
\item	
		The Hartshorne--Rao module
		$$H^1_\bullet(\mathcal{I}_C) \cong k(-1)$$
		is a  vector space of dimension $1$ concentrated in degree $1$.
\item The ideal sheaf $\mathcal I_C$ is $3$-regular.
\item The homogeneous ideal $I_C = H_\bullet^0(\mathcal I_C)$ has a $11$ generator in degree $3$.	
\item	
	The Hilbert numerator has shape
	$$(1-t)^4H_C(t)=1-11t^3 +18t^4-9 t^6+1.$$
	\end{itemize}
	Hence,  smooth maximal rank curves in $\mathcal H_{9, 4}$ have a Betti table as required. To establish that a general point in $C \in \mathcal H_{9, 4}$ is a maximal rank curve,  it suffices to produce a single maximal rank example. We will explicitely construct such examples. Moreover,  by inspection we find that the general $C$ lies on a smooth cubic hypersurface $X$.

Now let  $Z=\Bbb P^2(p_1, \ldots,  p_{10})$  is  a blow-up 
of $\Bbb P^2$ in $10$ points in general position by the complete  linear system
$$H_Z=(5;2^{2}, 1^{8})=5L-E_1-E_2-\sum\limits_{i=3}^{10}E_i.$$
Then $Z$ is a smooth surface of degree $9$ and sectional genus $4$ in $\Bbb P^6$. 
Let $q \in  \Bbb P^6$  be general point and   $\phi_q:\Bbb P^6 \dashrightarrow \Bbb P^5$  the projection from $q$ onto a hyperplane and let $Y_1 = \phi_q(Z) \subseteq\Bbb P^5$. Therefore,  $Y_1$ is a surface of degree $9$ and sectional genus $4$ and $K_{Y_1}^2=-1$.
Now we take a general hyperplane of $Y_1$,  we get that curve as requried.

The Betti table of the resolution of the section ring $\Gamma^\bullet(\mathcal O_C)$ as an $R$-module can be deduced with the same method,  since $H_{\Gamma^\bullet(\mathcal O_C)}(t) = H_C(t) + t$. 

\end{proof}

\begin{rem}

In the computation above,   we need exhibit explicit examples of modules,  curves or surfaces   defined over the rationals $\Bbb Q$ or complex numbers $\Bbb C$ satisfying some open conditions on their Betti numbers. Since such examples defined over an open part ${\rm Spec} \Bbb Z$,  the same result holds for algebraically closed fields of positive characteristic except for possibly finitely many primes $p$. Thus,  in principle,  one could try  to verify the result one by one for the remaining primes,  and then by semicontinuity,  we get the same result holds for the rationals $\Bbb Q$ or complex numbers $\Bbb C$.
In this paper,  we provide explicit constructions/equations of examples  with special geometric features. In an ancillary file (see \cite{TY20}) we have implemented our constructions in the computer algebra software Macaulay2 \cite{GSM2}.

\end{rem}

Now let $\mathcal W_{12, 10}$ be the irreducible component of the Hilbert scheme parametrizing projective Cohen-Macaulay smooth curves $C\subseteq \Bbb P^4$ of degree $12$,  and genus $10$
 which dominates the moduli space $\mathcal M_{10}$. 
\begin{cor}
Let $C\in \mathcal W_{12, 10}$ be a general point. 				 Then $C$ is of maximal rank and	
	the homogeneous coordinate ring $R_C=R/I_C$ and the section ring $\Gamma_\bullet(\mathcal{O}_C)$ have minimal free resolutions with the following	Betti tables:
	\begin{center}
	\begin{equation}\label{eqn1}
		\begin{tabular}{ c | c c c c c c}
			&    $0$ & $1 $ &$2 $  &$3 $\\ 
			\hline
			0&   $1    $ & $\cdot$ & $\cdot$ & $\cdot$ \\ 
			1&   $\cdot$ & $\cdot$ &$\cdot$ &$\cdot$ \\ 
			2&   $\cdot$ & $8   $ &$  9 $ &$  \cdot  $ \\ 
			3&   $\cdot$ & $\cdot   $ &$  \cdot $ &$  2  $ \\ 
		\end{tabular}	
		\end{equation}
	\end{center}
\end{cor}
\begin{proof}
As in Proposition \ref{C3.2},  we can compute the expected Betti tables of the $R$-resolutions of $R/I_C$ and $\Gamma_\bullet(\mathcal{O}_C)$. 
\end{proof}

\begin{cor}
A general cubic hypersurface in $\Bbb P^4$ contains a family of dimension $24$ of curves of genus $10$ and degree $12$.
\end{cor}
\begin{proof}
Let $C\in \mathcal W_{12, 10}$ be a general point and let $X$ be a general cubic threefold containing it. Then $\chi(\mathcal{N}_{C/X} ) = 12(5 - 3)=24$ by \cite[Lemma 5.3.]{ScT18}. We claim that $h^1(\mathcal{N}_{C/X} ) = 0$. It is sufficient to check this vanishing on one example,  as can be done with the Macaulay2.  
Hence,  $h^0(\mathcal{N}_{C/X}) = 24$. Let $\mathcal C$ be the space of cubic threefolds containing a general curve $C$ of genus $10$ and degree $12$,  up to projective equivalences. Since $h^0(\Bbb P4, \mathcal I_C(3))-1=8-1=7$. We have $$\dim\mathcal C =\dim \mathcal M_{10}+\rho(10, 4, 12)+7-24
=10, $$
where the Bril-Noether number $\rho(g, r, d):=g-(r+1)(g+r-d)$.

\end{proof}

\begin{prop}\label{4.2}
There exists a  projective  surface in $\Bbb P^5$ of degree $9$ sectional genus $4$ such that 	the homogeneous coordinate ring $R_Y=R/I_Y$ and the section ring $\Gamma_\bullet(\mathcal{O}_Y)$ have minimal free resolutions with the 	Betti tables as in Proposition \ref{C3.2}. In particular,  the minimal resolution of $\Gamma_\bullet(\mathcal{O}_Y)$ as a module over the homogeneous coordinate ring of a cubic fourfold $X \supset Y$ is eventually $2$-periodic with Betti numbers
		\begin{center}
		\begin{tabular}{ c | c c c c c c c c c}
			&    $0$ & $1 $ &$2 $  &$3 $& $4 $ &$5 $  &$6 $\\ 
			\hline
			0&  $1    $ & $\cdot$ & $\cdot$ & $\cdot$ & $\cdot$ & $\cdot$ & $\cdot$\\ 
			1&   $1$ & $5    $ &$1$&$\cdot$ & $\cdot$ & $\cdot$ & $\cdot$ \\ 
			2&   $\cdot   $ & $1    $ &$  9 $  & $ 9 $ & $\cdot$ & $\cdot$ & $\cdot$ \\ 
			3&   $\cdot$ & $\cdot$ &$  \cdot  $  & $  \cdot  $ & $  9 $ & $  9  $ & $\cdot$ 
		\end{tabular}		
	\end{center}

\end{prop}

\begin{proof}
Let $C$ be a smooth curve as in Proposition \ref{C3.2}. Then  it  has minimal free resolution

$$\xymatrix{o\ar[r] & {R^\prime}^{1}\ar[r] &{R^\prime}^{9}(-5)\ar[r]^{\psi}& {R^\prime}^{18}(-4)\ar[r]^{\varphi} & {R^\prime}^{11}(-3) \ar[r]& R^\prime\ar[r]&R_{C} \ar[r]&0}, $$

where ${R^\prime}=k[x_0, \ldots, x_4]$.
Recall that we say that a sequence of $R^\prime$-modules

$$\xymatrix{\ldots\ar[r]&M_{i+1}\ar[r]^{\varphi_{i+1}}&M_i\ar[r]^{\varphi_i}&M_{i-1}\ar[r]&\ldots}$$
is exact at $M_i$ if and only if  $\mathrm{Im }\  \varphi_{i+1}=\mathrm{Ker }\ \varphi_{i}$. The sequence is exact if it is exact at every $M_i$ in the sequence.  Now we will construct  surface $Y$ as follows. 
\begin{itemize}

\item Let $A$ be an $11\times 18$ matrix of indeterminates and $B$ be an $18\times 9$ matrix of indeterminates.  Let $T=k[A, B]$ be the polynomial ring generated by all the indeterminates $A_{ij}$ and $B_{ij}$ and $T^\prime=R^\prime\otimes T$. Consider  a sequence $F_\bullet$ of $T^\prime$-modules 
$$F_\bullet:\xymatrix{ {T^\prime}^{9}(-5) \ar[r]^{\psi+B} &{T^\prime}^{18}(-4)\ar[r]^{\varphi+A}& {T^\prime}^{11}(-3).}$$

A sequence $F_\bullet$ of $T^\prime$-modules  is exact if and only if 
$$A\circ \psi+\varphi \circ B\equiv 0 \text{ and } AB\equiv 0 $$
\item Since $\psi$ and $\varphi$ is linear matrixes over polynomial ring $R^\prime$,  the condition $A\circ \psi+\varphi \circ B\equiv 0 $ gives a linear  relationship on the indeterminates $A_{ij}$ and $B_{ij}$. Let $m+1$ is numbers of independent variables on   such linear  relationship.  Let $S=k[y_0, \ldots, y_m]$. Then the condition $AB\equiv 0 $ defined a varieties $V$ such that a point of $V$ is corespondent to  a smooth projective  surface $Y(p)$ of degree $9$ sectional genus $4$. In the current case,  they are $m=14$ and $V$ is a cone  in $\Bbb  P^{14}$ of dimensional $8$ generated by $23$ quadric elements. 
\end{itemize}
Hence,  we get that surface $Y$ has minimal free resolutions with the Betti tables as in Proposition \ref{C3.2}. Moreover,  the Hartshorne--Rao module: 
		$$H^1_\bullet(\mathcal{I}_Y) \cong k(-1)$$
		is a  vector space of dimension $1$ concentrated in degree $1$.
Since $H_{\Gamma^\bullet(\mathcal O_Y)}(t) = H_Y(t) + t$,  the Betti table of the resolution of the section ring as an $R$-module is as in Proposition \ref{C3.2}.
\end{proof}

Recall that a smooth cubic fourfold  $X$ in $\Bbb P^5$  is {\it special},  if there is  an embedding of a saturated rank-$2$ lattice 
$$L_\delta:=\langle h^2, Y\rangle\hookrightarrow A(X), $$
where $A(X)$ is the lattice of middle Hodge classes,  $h\in {\rm Pic}(X)$  is the hyperplane class,   $Y$ is an algebraic surface not homologous to a complete intersection,  and $\delta$ is the determinant of the intersection matrix of $\langle h^2, Y\rangle$. In 2000,  Hassett showed in \cite{Has00} that the locus  $\mathcal C_\delta$ of special cubic fourfolds  $X$ of a discriminant $\delta$ is an irreducible divisor which is nonempty if and only if $\delta>6$ and $\delta\equiv 0, 2 \pmod 6 $.  
Now let $\mathcal H_{12, 10}$ be the irreducible component of the Hilbert scheme parametrizing projective Cohen-Macaulay smooth surfaces $Y\subseteq \Bbb P^5$ of degree $12$,  and sectional genus $10$ which dominates the moduli space $\mathcal W_{12,10}$. With the above notation, we have the following proposition.

\begin{prop}\label{P3.2}
	Let $Y\in \mathcal H_{12, 10}$ be a general point.				 Then $Y$ is of  maximal rank and	
	the homogeneous coordinate ring $R_Y=R/I_Y$ and the section ring $\Gamma_\bullet(\mathcal{O}_Y)$ have minimal free resolutions with the Betti tables \ref{eqn1}.
	Moreover,  if $X$ is a general cubic fourfold containing $Y$ then $X\in \mathcal C_{18}$.

\end{prop}
\begin{proof}
	Assuming that the restriction map $H^0(\mathcal{O}_{\Bbb P^5}(m))\to H^0(\mathcal{O}_{Y}(m))$ has maximal rank. Then we have the following statements.
	\begin{itemize}

\item	The Hilbert series of the homogeneous coordinate ring of $Y$ is
	$$H_Y(t) = 1 + 6t + 21t^2 + 48 t^3  + 87 t^4  + 138 t^5  +201 t^6 +276 t^7 +\ldots.$$ 
\item	
	The Hartshorne--Rao module 
	$$H^1_\bullet(\mathcal{I}_Y)=H^2_\bullet(\mathcal{I}_Y)=0.$$
\item	
	The Hilbert numerator has shape
	$$(1-t)^5H_Y(t)=1-8t^3 +9t^4-2 t^6.$$
	\end{itemize}
	
		Since $H_Y(t)=H_{\Gamma_\bullet(\mathcal{O}_Y)}(t)$,  the above Betti table  of the resolution $F_\bullet$ of the section ring $\Gamma_\bullet(\mathcal{O}_Y)$ as an $R$-module.
	Finally,  we compute the Betti number of $\Gamma_\bullet(\mathcal{O}_Y)$ as an $R_X$-module. The
	(possibly non-minimal) resolution  has Betti table as above.

To show that the Betti tables are indeed the expected ones and that,  a posteriori,  a general surface $Y$ is of maximal rank,  we only need to exhibit a concrete example,  which we construct  from Theorem \ref{main1} and Proposition \ref{4.2}.	

Now let  $X$ be a general cubic fourfold containing $Y$. 		By computing the self-intersection of $Y \subseteq X$ using the formula from \cite{Has00}, 
$$Y^2 =c_2(\mathcal N_{Y/X})=6H^2 +3H\cdot K +K^2 -\chi_{top}=54.$$
Hence,  we have $\delta=3Y^2-d^2=3\cdot 54-12^2=18$. Hence $[X]\in \mathcal C_{18}$.

\end{proof}

\begin{thm}\label{main2}
Let  $X$ be a general special cubic fourfold in $\mathcal C_{18}$. Then
 the moduli space of stable rank $3$ Ulrich bundles on a special cubic fourfold  $X$ with first Chern class $c_1 = 3H$,  where $H$ is the hyperplane class,  and $c_2 = 12$ is nonempty and smooth of dimension $10$. 
\end{thm}
\begin{proof}
By Proposition \ref{P3.2},  moduli space of stable rank $3$ Ulrich bundles on a special cubic fourfold  $X$  is nonempty. Now we will show that a general $[X] \in \mathcal C_{18}$ contains a surface $[Y]\in \mathcal H_{12,10}$. In fact,  $\dim \mathcal H_{12,10}=63$ and $\mathcal H_{12,10}$ is generically smooth. Indeed,  one can verify that $h^1(\mathcal{N}_{Y/\Bbb P^5}) = 0$ for $[Y]\in \mathcal H_{12,10}$.  Let $\mathcal{C} \subseteq |H^0(\Bbb P^5(3)) | \cong \Bbb P^{55}$ be the open set corresponding to smooth cubic hypersurfaces. Since $h^0(\mathcal I_Y(3)) = 5$ so that the locus
$\mathcal{D}_{18} = \{([Y],  [X]) : Y\subseteq X\} \subseteq \mathcal{H}\times \mathcal{C}$
has dimension $63 + 8 -1 = 70$. The image of $\pi_2 : \mathcal{D}_{18} \to \mathcal{C}$ has dimension at most $54$ because the general cubic does not contain any $Y$ belonging to $ \mathcal H$. For every $[X] \in  \pi_2(\mathcal{D}_{18})$ we have
$$\dim(\pi^{-1}_2 ([X]))
 \ge \dim(\mathcal{D}_{18}) - \dim(\pi_2(\mathcal{D}_{18}))  \ge 70 - 54 = 16.$$
Since $h^0(\mathcal{N}_{Y/X}) \ge \dim_{[Y\subset X]}(\pi^{-1}
_2 ([X]))$ for every $[Y\subset X] \in \pi^{-1}_2 ([X])$,  to show that a general
$X \in \mathcal{C}_{18}$ contains a surface $Y$, it is sufficient to verify that $h^0(\mathcal{N}_{Y/X}) = 16$ for a general $Y$
and for a smooth $X \in |H^0(\mathcal{I}_Y(3))|$.
We verified this via Maucaulay2 and we can conclude that $h^0(\mathcal{N}_{Y/X})=16$.
	
	 To show that the moduli space is smooth of dimension $10$,  we just compute $h^1(\mathcal F\otimes\mathcal  F^\vee) = 10$ and $h^2(\mathcal F\otimes\mathcal  F^\vee) = h^3(\mathcal F\otimes\mathcal  F^\vee) = 0$. This is elementary (see the proof of Theorem \ref{main3}).
	
\end{proof}

\begin{rem}
$1)$ As an application of the result above and \cite[Theorem 4.3]{KuM09},  one gets that the smooth locus of any moduli space of stable rank $3$ Ulrich bundles on a special cubic fourfold  $X\in\mathcal C_{18}$ carries a closed symplectic form.\\
$2)$ The Hilbert scheme parametrizing projective Cohen-Macaulay smooth surfaces of degree $12$, and sectional genus $10$ has least two components. In fact,
let $X$ be a cubic fourfold containing an elliptic ruled
surface $T$ as in Theorem 2 in \cite{AHTV19}. Consider $\Gamma_\bullet(\mathcal{O}_T)$
	as a $R_X$-module,  the periodic part of its minimal free resolution
	yields,  up to twist,  matrix factorization of the form
	$$\xymatrix{ R^9(-4)\oplus R(-3)\ar[r]^{\psi}& R^{9}(-2)\oplus R(-3)\ar[r]^{\quad\varphi} &R^{9}(-1)\oplus R.}$$	
Thus there exists an ACM surface $Y$ such that
	homogeneous coordinate ring $R_Y=R/I_Y$ and the section ring $\Gamma_\bullet(\mathcal{O}_Y)$ have minimal free resolutions with the Betti tables
	\begin{center}
		\begin{tabular}{ c | c c c c c c}
			&    $0$ & $1 $ &$2 $  &$3 $\\ 
			\hline
			0&   $1    $ & $\cdot$ & $\cdot$ & $\cdot$ \\ 
			1&   $\cdot$ & $\cdot$ &$\cdot$ &$\cdot$ \\ 
			2&   $\cdot$ & $8   $ &$  9 $ &$  1 $ \\ 
			3&   $\cdot$ & $\cdot   $ &$  1 $ &$  2  $ \\ 
		\end{tabular}	
	\end{center}
	
\end{rem}

Now, for the proof of Theorem \ref{main3} we need the following result.
\begin{lem}\label{exR23}
There exists a smooth cubic fourfold having both Ulrich bundles of rank $2$ and $3$.
\end{lem}
\begin{proof}
Let $p_1,\ldots,p_{10}$ be points in general position in $\Bbb P^2$ and $C_0$ a quadric curve passing four points $p_3,\ldots,p_6$. Let  $Z=\Bbb P^2(p_1, \ldots,  p_{10})$  is  a blow-up 
of $\Bbb P^2$ in $10$ points in general position by the complete  linear system
$H_Z=(5;2^{2}, 1^{8})=5L-E_1-E_2-\sum\limits_{i=3}^{10}E_i.$
Then $Z$ is a smooth surface of degree $9$ and sectional genus $4$ in $\Bbb P^6$. 
Let $C_1$ be the strict transform of $C_0$ on $Z$. Then $C_1$ is a smooth rational curve of degree $6$.

Let $q$  be a point belong a trisecant line to the surface $Z$ and   $\phi_q:\Bbb P^6 \dashrightarrow \Bbb P^5$  the projection from $q$ onto a hyperplane and let $Y = \phi_q(Z) \subseteq\Bbb P^5$. Therefore,  $Y$ is a surface of degree $9$ and sectional genus $4$ and $K_{Y}^2=-1$. Moreover, $Y$  has the minimal free resolution with the Betti tables as in Proposition \ref{C3.2}. 

On the other hand, $C_2 = \phi_q(C_1) \subseteq\Bbb P^5$ is a smooth curve on $Y$ with the minimal free resolution
	\begin{center}
		\begin{tabular}{ c | c c c c c c}
			&    $0$ & $1 $ &$2 $  &$3 $&$4$&$5$\\ 
			\hline
			0&   $1    $ & $\cdot$ & $\cdot$ & $\cdot$& $\cdot$& $\cdot$ \\ 
			1&   $\cdot$ & $8$ &$11$ &$3$ & $\cdot$& $\cdot$\\ 
			2&   $\cdot$ & $\cdot   $ &$  4 $ &$  10 $ &$6$&$1$ \\ 
		\end{tabular}	
	\end{center}
Then $C_2$ is contained in a smooth rational normal scroll surface $S_1$ of degree $4$.  Moreover $S$ is contained in a smooth rational normal scroll cubic threefold $\Sigma$.
Let $X$ be a smooth cubic fourfold which contains both surfaces $Y$ and $S_1$. Then there exists a quintic del Pezzo surface $S$ such that
$$X\cap \Sigma=S\cup S_1.$$ 
Then by Theorem \ref{class100} and Proposition \ref{4.2}, $X$ has Ulrich bundles of rank $2$ and $3$.
Notice that, in an ancillary file (see \cite{TY20}), we provide the explicit homogeneous ideal of  surfaces $Y$, $S$ in the cubic fourfold $X$.
\end{proof}
\begin{rem}
In the proof of Lemma \ref{exR23}, we give an explicit example of a triple $S$, $Y \subset X$ over a finite field with the help of Macaulay2 and then establishing the lemma in characteristic $0$ with semi-continuity.  By Theorem \ref{class100} and \ref{main2}, $[X]\in \mathcal C_{14}\cap \mathcal C_{18}$. Since $S$ and $Y$ intersect transversally in $16$ points, $X$ is a rational cubic fourfold in $\mathcal C_{18}$ associated with a good sextic del Pezzo fibration(see \cite[Definition 11]{AHTV19}) and with {\it one nontrivial Brauer class.} 
Notice that intersection $\mathcal C_{14}\cap \mathcal C_{18}$ has $9$ irreducible components and  H. Awada also gave an other  example of a rational cubic fourfold in $\mathcal{C}_{18}$ associated with a good sextic del Pezzo fibration and with one nontrivial Brauer class(see \cite{Awa19}).
\end{rem}

\begin{proof}[\sl Proof of Theorem \ref{main3}]
 Let $X$ be as in Lemma \ref{exR23}. Then $[X]\in\mathcal C_{14}\cap \mathcal C_{18}$.
Since $\mathcal F$ is an Ulrich bundle of rank $r \ge 2$,   it has the following minimal free resolution in $\Bbb P^5$
$$\xymatrix{0\ar[r]&\mathcal O_{\Bbb P^5}^{3r}(-1) \ar[r]^{\varphi}&\mathcal O_{\Bbb P^5}^{3r} \ar[r]&\mathcal F\ar[r]&0}$$
by \cite[Proposition 3.7]{CaH11}. Tensoring with $F^\vee$,  we get a right exact sequence on $X$, 
$$\xymatrix{0\ar[r]&\mathcal F^\vee(-1)^{3r} \ar[r]^{\varphi}&\mathcal {F^\vee}^{3r} \ar[r]&\mathcal F\otimes \mathcal F^\vee\ar[r]&0.}$$
Now $\mathcal F^\vee(2)$ is an Ulrich bundle,  so $\mathcal F^\vee$ and $\mathcal F^\vee(-1)$ have no cohomology. It follows from cohomology sequences on $X$,  that $H^2(\mathcal F\otimes \mathcal F^\vee) = H^3(\mathcal F\otimes \mathcal F^\vee) = 0$.  Therefore,  the moduli space is smooth. Furthermore,   $\chi(\mathcal F\otimes\mathcal  F^\vee) = \chi(\mathcal F_H\otimes\mathcal  F^\vee_H) = -r^2$ by \cite[Corollary 2.13]{CaH11}, where $\mathcal F_H$ is the general hyperplane section of $\mathcal F$,  since $c_1(\mathcal F) = rH$.
For $\mathcal F$ stable or simple we have $h^0(\mathcal F\otimes\mathcal  F^\vee) = 1$,  and $h^2(\mathcal F\otimes\mathcal  F^\vee) = h^3(\mathcal F\otimes\mathcal  F^\vee) = 0$,  so $h^1(\mathcal F\otimes\mathcal  F^\vee) = r^2 + 1$ is the dimension of the moduli space.

It remains to show the existence. We proceed by induction on $r$,  the cases $r = 2$,  $3$ by Lemma \ref{exR23}. So let $r \ge 4$,  and choose $\mathcal E$ stable of rank $2$,  and $\mathcal F$ stable of rank $r-2$,  different from $\mathcal E$. Then $h^i(\mathcal E \otimes\mathcal F^\vee) = 0$ for $i = 0$,  $2$,  $3$,  so $h^1(\mathcal E \otimes\mathcal F^\vee) = -\chi(\mathcal E \otimes\mathcal F^\vee) = -\chi(\mathcal E_H \otimes\mathcal F^\vee_H) = 2(r-2)$ by \cite[Corollary 2.13]{CaH11}. In particular,  this number is positive,  so there exist non-split extensions
$$\xymatrix{0\ar[r]&\mathcal E\ar[r]&\mathcal G\ar[r]&\mathcal F\ar[r]&0, }$$
and the new bundle $G$ will be a simple Ulrich bundle of rank $r$. We consider the modular family of these simple bundles,  which will be smooth of dimension $r^2 + 1$ by the above observations.
If the general simple bundle in this family is not stable,  it must have the same splitting type as the ones just constructed. However,  the dimension of the family of extensions above is
$$\dim{\mathcal E} + \dim\mathcal{F} + \dim({\rm Ext}^1(\mathcal F,  \mathcal E)) - 1 =  r^2 - 2r + 5.$$
Since $r \ge 4$,  this number is strictly less than $r^2 + 1$. We conclude that the general simple bundle of rank $r$ is stable,  so stable bundles exist.
\end{proof}




For a very general hypersurface $X \subseteq \Bbb P^3$ of degree at least $4$,  the Noether-Lefschetz theorem says that every curve $C \subset X$ is a complete intersection of $X$ with a surface in $\Bbb P^3$,  i.e. $C = X \cap S$ where $S \subset \Bbb P^3$ is a surface. As a consequence of this theorem,  any ACM line bundle on such an $X$ is the restriction of a line bundle on $\Bbb P^3$.  
One might ask whether this generalizes in some way to higher codimension  and higher dimensional hypersurfaces. Motivated by this,  Griffiths and Harris  conjectured whether subvarieties $Y$ of codimension two of a hypersurface $X\subset \Bbb P^n$ can be obtained by intersecting with a subvariety of codimension two of the ambient space $\Bbb P^n$. 
We shall call subvarieties $Y \subset X\subset \Bbb P^n$ which are not intersections of $X$ with any subvariety of codimension two of  $\Bbb P^n$ as {\it distinguished}. C. Voisin very soon  proved that  a  general threefold $X \subset \Bbb P^4$ always contains distinguished curves $C \subset X$,  thus proving that this conjecture is false. In \cite{KRR09},  it is shown that there exists a large class of distinguished  ACM subvarieties in smooth hypersurfaces of dimension at least three and degree at least two. But such ACM subvarieties may not give smooth ones. The next goal in this section is to provide  large families of distinguished  ACM subvarieties  on special cubic fourfolds.

\begin{thm}\label{main5}	
Let $\mathcal F$ be an Ulrich bundle on a cubic foufold $X$ and $Y$  a surface  constructed from $\mathcal F$ as in Theorem \ref{thm4.5}. Then $Y$ is a distinguished  ACM subvarieties on $X$.
\end{thm}
\begin{proof}
It follows immediately that Proposition 4 in \cite{KRR09} and Theorem \ref{main1}.

\end{proof}



  \begin{thm} \label{main6}
  Let $X$ be a very general cubic fourfold in $\Bbb P^5$ containing a plane. Then the moduli space of stable rank $4$ Ulrich bundles on a cubic fourfold $X$ is nonempty and smooth of dimension $ 17$.
  \end{thm}
 \begin{proof} 
 The argument can be divided in a few parts.\\
{\bf Step 1:}   {\it There exists stable rank $4$ Ulrich bundles $\mathcal F$ on a cubic fourfold $X$.}  Let   $Z=\Bbb P^2(p_1, \ldots,  p_{10})$ is  a blow-up 
of $\Bbb P^2$ in $10$  points in general position by the complete  linear system
$$H_Z=(6;2^{5}, 1^5)=6L-\sum\limits_{i=1}^{5}2E_i-\sum\limits_{i=6}^{10}E_i.$$
Then $Z$ is smooth surface of degree $11$ and sectional genus $5$ in $\Bbb P^7$. Let $\ell\subset  \Bbb P^7$  be general line,   and  $\phi_{\ell}:\Bbb P^7 \dashrightarrow \Bbb P^5$ be the projection from
$\ell$ onto a hyperplane. Put ${Y_1} = \phi_\ell(Z) \subseteq\Bbb P^5$. Therefore,  ${Y_1}$ is a surface of degree $11$ and sectional genus $5$ and $K_{Y_1}^2=-1$.
Moreover,  homogeneous coordinate ring $R_{Y_1}=R/I_{Y_1}$ have the following Betti table 
\begin{center}
		\begin{tabular}{ c | c c c c c c c}
			&   & $0$ & $1 $ &$2 $  &$3 $&$4$&$5$\\ 
			\hline
			0&  & $1    $ & $\cdot$ & $\cdot$ & $\cdot$&$\cdot$&$\cdot$ \\ 
			1&  & $\cdot    $ & $\cdot$ & $\cdot$ & $\cdot$&$\cdot$&$\cdot$ \\ 
			2&  & $\cdot$ & $1$ &$\cdot$ &$\cdot$&$\cdot$ &$\cdot$\\ 
			3&  & $\cdot$ & $25   $ &$  65 $ &$  63  $ &$28$&$5$\\ 
		\end{tabular}	
		\ \ \
			\begin{tabular}{ c | c c c c c c}
			&   & $0$ & $1 $ &$2 $  &$3 $\\ 
			\hline
			0&  & $1    $ & $\cdot$ & $\cdot$ & $\cdot$ \\ 
			1&  & $2$ & $7$ &$\cdot$ &$\cdot$ \\ 
			2&  & $\cdot$ & $1   $ &$  10 $ &$  5  $ \\ 
		\end{tabular}

\end{center}

Let $X$  be a unique smooth cubic fourfold containing the surface ${Y_1}$. 
Consider $\Gamma_*(\mathcal{O}_{Y_1})$
	as a $R_X$-module,  the periodic part of its minimal free resolution
	yields,  up to twist,  matrix factorization of the form
	$$\xymatrix{R^{12}_X& R^{12}_X(-1)\ar[l]_{\psi}& \ar[l]_{ \varphi}R^{12}_X(-3).}$$	
Let $F^\bullet$ and $\overline{G^\bullet}$ be minimal free resolutions of the section ring $\Gamma_*(\mathcal{O}_{Y_1})$ as an $R$-module and an $R_X$-module, respectively. We  denote   by  $\varphi$ the syzygy map $\xymatrix{\overline{G_3} &\ar[l] \overline{G_4}}$ and $\mathcal{F} = \widetilde{{\rm coker} \varphi}(-3)$. Then $\mathcal F$ is rank $4$ Ulrich bundles $\mathcal F$ on the cubic fourfold $X$.  

{\bf Step 2:} {\it  If $X$ is very general cubic fourfold in $\Bbb P^5$ containing a plane,  then $X$ contains a surface $Y$ of degree $22$ as in Theorem \ref{main1}.}  Since $\mathcal F$ is rank $4$ Ulrich bundles $\mathcal F$ on the cubic fourfold $X$,  by   Theorem \ref{main1},  there exists an ACM surface $Y$ of degree $22$  as in Theorem \ref{main1}. Moreover,  the irreducible component of the Hilbert scheme parametrizing the surfaces $Y \subset \Bbb P^5$ has dimension $91$ and it is generically smooth. Indeed,  one can verify that $h^1(\mathcal{N}_{Y/\Bbb P^5}) = 0$.  Let $\mathcal{C} \subseteq |H^0(\Bbb P^5(3)) | \cong \Bbb P^{55}$ be the open set corresponding to smooth cubic hypersurfaces. Since $h^0(\mathcal I_Y(3)) = 1$ so that the locus
$\mathcal{D}_{8} = \{([Y],  [X]) : Y\subseteq X\} \subseteq \mathcal{H}\times \mathcal{C}$
has dimension $91 + 1 -1 = 91$. The image of $\pi_2 : \mathcal{D}_{8} \to \mathcal{C}$ has dimension at most $54$ because the general cubic does not contain any $Y$ belonging to $ \mathcal H$. For every $[X] \in  \pi_2(\mathcal{D}_{8})$ we have
$$\dim(\pi^{-1}_2 ([X]))
 \ge \dim(\mathcal{D}_{8}) - \dim(\pi_2(\mathcal{D}_{8}))  \ge 91 - 54 = 37.$$
Since $h^0(\mathcal{N}_{Y/X}) \ge \dim_{[Y\subset X]}(\pi^{-1}
_2 ([X]))$ for every $[Y\subset X] \in \pi^{-1}_2 ([X])$,  to show that a general
$X \in \mathcal{C}_{8}$ contains a surface $Y$ it is sufficient to verify that $h^0(\mathcal{N}_{Y/X}) = 37$ for a general $Y$
and for a smooth $X \in |H^0(\mathcal{I}_Y(3))|$,  see also \cite{Nue17} for a similar argument.
We verified this via Maucaulay2 and we can conclude that $h^0(\mathcal{N}_{Y/X})=37$.
Since $[X] \in \mathcal C_8$ if and only if $X$ contains a plane,  we get statement as required.    


{\bf Step 3:}  {\it Dimension and stableness}. To complete the diminsioness statement of Theorem,  we should prove that  for stable rank $4$ Ulrich bundles $\mathcal F$ on $X$,  we have $h^1(\mathcal F\otimes\mathcal  F^\vee) = 17$ and $h^2(\mathcal F\otimes\mathcal  F^\vee) = h^3(\mathcal F\otimes\mathcal  F^\vee) = 0$. For this we can use the same deformation argument as in the proof of Theorem \ref{main3}. 

By Theorem \ref{main2},  there are no rank $2$ or $3$ Ulrich bundles on a very general  cubic fourfold $X$ containing a plane $P$. Therefore,  as observed in \cite[Section 5]{CaH12},  $\mathcal F$ is stable. 


\end{proof}

\begin{cor}	
	Let $X$ be a very general cubic fourfold in $\Bbb P^5$ containing a plane  or a general special cubic fourfold in $\mathcal C_{14}$ or $\mathcal C_{18}$. Then $X$  contains distinguished smooth surfaces $Y$.
\end{cor}
\begin{proof}
These facts follow immediately from Theorems \ref{main1},  \ref{main2},  \ref{main3},  \ref{main5} and \ref{main6}
\end{proof}

\end{document}